\documentclass{gtpart}     
\usepackage{pinlabel}  
\usepackage{graphicx}  
\usepackage[all]{xy}
\title{A surgical perspective on quasi-alternating links}

\author{Liam Watson}
\givenname{Liam}
\surname{Watson}
\address{Department of Mathematics, UCLA, 520 Portola Plaza, Los Angeles, CA 90095.}
\email{lwatson@math.ucla.edu}
\urladdr{http://www.math.ucla.edu/~lwatson}


\arxivreference{}
\arxivpassword{}

\volumenumber{}
\issuenumber{}
\publicationyear{}
\papernumber{}
\startpage{}
\endpage{}
\doi{}
\MR{}
\Zbl{}
\received{}
\revised{}
\accepted{}
\published{}
\publishedonline{}
\proposed{}
\seconded{}
\corresponding{}
\editor{}
\version{}

\newtheorem{theorem}{Theorem}[section]    
\newtheorem{proposition}{Proposition}[section]
\makeatletter 
\let\c@proposition=\c@theorem 
\makeatother 

\theoremstyle{definition}
\newtheorem{definition}[theorem]{Definition}    
\theoremstyle{definition}             
\theoremstyle{definition}
\newtheorem*{example}{Example} 

\usepackage{amsmath,amsfonts,amssymb,verbatim}

\newcommand{\OSz}{Ozsv\'ath and Szab\'o} 
\newcommand{\QA}{{\bf QA}}

\newcommand{\sC}{\mathcal{C}}

\newcommand{\sQ}{\mathcal{Q}}

\newcommand{\bZ}{\mathbb{Z}}
\newcommand{\bQ}{\mathbb{Q}}

\newcommand{\pq}{\frac{p}{q}}

\newcommand{\overzero}{\frac{1}{0}}

\newcommand{\lm}{\lambda_M}

\newcommand{\into}{\hookrightarrow}

\newcommand{\Br}{\boldsymbol\Sigma}
\newcommand{\fibre}{\varphi}

\newcommand{\HFhat}{\widehat{\operatorname{HF}}}

\newcommand{\Rk}{\operatorname{rk}}

\newcommand{\positive}
	{\raisebox{-2pt}{\includegraphics[scale=0.085]{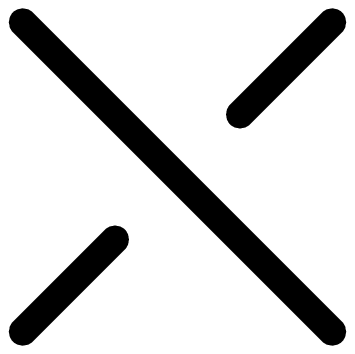}}}
\newcommand{\zero}
	{\raisebox{-2pt}
	{\includegraphics[scale=0.085]{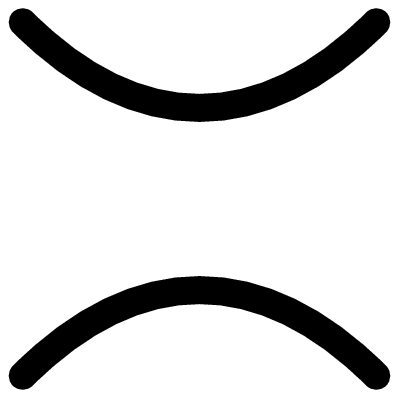}}}
\newcommand{\one}
	{\raisebox{-2pt}
	{\includegraphics[scale=0.085]{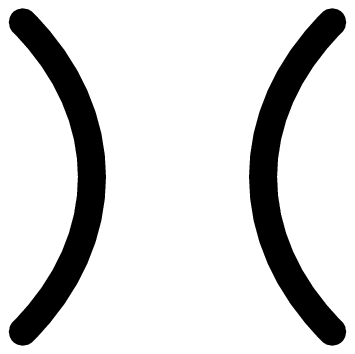}}}

\begin{document}

\begin{abstract}   
We show that quasi-alternating links arise naturally when considering surgery on a strongly invertible L-space knot (that is, a knot that yields an L-space for some Dehn surgery). In particular, we show that for many known classes of L-space knots, every sufficiently large surgery may be realized as the two-fold branched cover of a quasi-alternating link. Consequently, there is considerable overlap between L-spaces obtained by surgery on $S^3$, and L-spaces resulting as two-fold branched covers of quasi-alternating links. By adapting this approach to certain Seifert fibered spaces, it is possible to give an iterative construction for quasi-alternating Montesinos links. \end{abstract}

\maketitle



%

This article studies quasi-alternating links -- in particular how such links arise naturally from the perspective of Dehn surgery -- and constitutes an extension of certain aspects of the author's talk at the International Georgia Topology Conference. The talk discussed the application of Khovanov homology as an obstruction to exceptional surgeries on strongly invertible knots \cite{Watson2008}, and central to this work is the following: if a two-fold branched cover of $S^3$, branched over a link $L$, has finite fundamental group then the reduced Khovanov homology of the branch set $L$ is supported in at most 2 diagonals \cite[Theorem 5.2]{Watson2008}. In establishing this result, one is led to consider Dehn surgery on certain Seifert fibered manifolds, the result of which is in turn often forced to branch cover a quasi-alternating link for topological reasons. More generally, many families of quasi-alternating links arise in this way, and it is this point that we hope to elucidate here.


\section{Background and context}

Quasi-alternating links were introduced by \OSz\ as a natural extension of the class of alternating links in the context of Heegaard-Floer homology for two-fold branched covers \cite{OSz2005-branch}. 

\begin{definition}\label{def:qa}The set of quasi-alternating links $\sQ$ is the smallest set of links containing the trivial knot, and closed under the following relation: if $L$ admits a projection with distinguished crossing $L(\positive)$ so that \[\det(L(\positive))=\det(L(\zero))+\det(L(\one))\] for which $L(\zero),L(\one)\in\sQ$, then $L=L(\positive)\in\sQ$ as well. \end{definition}

In particular, \OSz\ demonstrate that a non-split alternating link is quasi-alternating \cite[Lemma 3.2]{OSz2005-branch}, and that the two-fold branched cover of a quasi-alternating link is an L-space \cite[Proposition 3.3]{OSz2005-branch}. Recall that a rational homology sphere $Y$ is an L-space if $\Rk\HFhat(Y)=|H_1(Y;\bZ)|$, where $\HFhat$ denotes the (hat version of) Heegaard-Floer homology \cite{OSz2005-lens}.  

In a related vein, Manolescu demonstrated that the rank of the knot Floer homology of a quasi-alternating knot is given by the determinant of the knot \cite{Manolescu2007}, and subsequently Manolescu and Ozsv\'ath showed that quasi-alternating links have both Khovanov \cite{Khovanov2000} and knot Floer \cite{OSz2004-knot,Rasmussen2003} homology groups supported in a single diagonal (or, {\em thin} homology) \cite{MO2007}. The same is true of odd-Khovanov homology \cite{ORS2007}, and as a result these homology theories contain essentially the same information as their respective underlying polynomials when applied to  quasi-alternating links. While it seemed possible that the collection of thin links was the same as the collection of quasi-alternating links, the situation is in fact more complicated. Greene has demonstrated that the knot $11^n_{50}$ is not quasi-alternating despite having all of the aforementioned homologies supported in a single diagonal \cite{Greene2009}. 

L-spaces turn out to have interesting topological properties, and as such it would be interesting to be able to characterize such manifolds without reference to Heegaard-Floer homology. Since quasi-alternating links give rise to L-spaces (by way of two-fold branched covers), attempts to better understand such branch sets may be viewed as an approach to such a characterization. In constructing the spectral sequence from the reduced Khovanov homology of a branch set, converging to the Heegaard-Floer homology of the two-fold branched cover \cite{OSz2005-branch}, Ozsv\'ath and Szab\'o show that the skein exact triangle in Khovanov homology lifts to the surgery exact sequence for Heegaard-Floer homology in the cover. With this in mind, our aim is to better understand quasi-alternating knots from the perspective of Dehn surgery.  

Let $S^3_r(K)$ denote $r$-surgery on a knot $K\into S^3$. If $S^3_r(K)$ is an L-space then $K$ is referred to as an L-space knot, and it is well known that every such knot gives rise to an infinite family of L-spaces: if $S^3_{r_0}(K)$ is an L-space for $r_0>0$, then $S^3_r(K)$ is an L-space as well for every rational number $r\ge r_0$ \cite{OSz2005-lens}. It seems natural to ask then if quasi-alternating links enjoy an analogous property. 

To this end, we will be primarily interested in the knots $K\into S^3$ having the following property:
\begin{itemize}
\item[{\bf QA}] There is a positive integer $N$ such that $S^3_{r}(K)$ is the two-fold branched cover of a quasi-alternating link for all rational numbers $r\ge N$.
\end{itemize}
Notice that a knot satisfying property \QA\ is necessarily strongly invertible. It turns out that many L-space knots are strongly invertible, and our aim is to establish that these examples satisfy property \QA.  
\begin{theorem} The following knots have property \QA: 
\begin{itemize}
\item[1.] Torus knots and, more generally, Berge knots (see \fullref{prp:berge}, as well as \cite[Section 4]{Watson2008}). 
\item[2.] The $(-2,3,q)$-pretzel knot, for all positive, odd integers $q$ (see \fullref{thm:pretzels}). 
\item[3.] All sufficiently positive cables of knots with property \QA\ (see \fullref{thm:cables}).  
\end{itemize}
\end{theorem}

Consequently, there is considerable overlap between known L-spaces obtained by surgery on a knot in $S^3$, and L-spaces obtained as the two-fold branched cover of a quasi-alternating link.

In certain instances, one may ask if property \QA\ holds for knots in manifolds other than $S^3$ (that is, if suitably large surgeries on the knot may be realized as two-fold branched covers of quasi-alternating links). For example, it can be shown that Dehn surgery on a regular fibre of certain Seifert fibered spaces gives rise to manifolds that are two-fold branched covers of quasi-alternating links, and this in turn yields an iterative  geometric construction for producing all quasi-alternating Montesinos links (see \fullref{sec:QA-Montesinos}). 

\section{The Montesinos trick}

A tangle $T$ is a pair $(B^3,\tau)$ where $\tau\into B^3$ is a pair of properly embedded arcs in a three-ball (together with a possibly empty collection of closed components), meeting the boundary sphere transversally in the 4 endpoints. The two-fold branched cover of a tangle will be denoted by $\Br(B^3,\tau)$, where $\tau$ is the branch set. Notice that $\Br(B^3,\tau)$ is a manifold with torus boundary. Montesinos' observation is that a Dehn filling of such a manifold may be viewed as a two-fold branched cover of $S^3$, the branch set of which is obtained by attaching a rational tangle to $T$ in a prescribed way \cite{Montesinos1975}. Recall that a tangle is rational if and only if the two-fold branched cover is a solid torus, and that tangles in this setting are considered up to homeomorphism of the pair $(B^3,\tau)$.  

To exploit this fact, generally referred to as {\em the Montesinos trick}, we briefly recall the notation introduced in \cite{Watson2008}. 

\begin{figure}[ht!]
\begin{center}
\labellist \small
	\pinlabel $\gamma_{\overzero}$ at -73 278
	\pinlabel $\gamma_0$ at -150 217
	\pinlabel $\tau(\overzero)$ at 280 210
	\pinlabel $\tau(0)$ at 615 210
\endlabellist
\raisebox{0pt}{\includegraphics[scale=0.25]{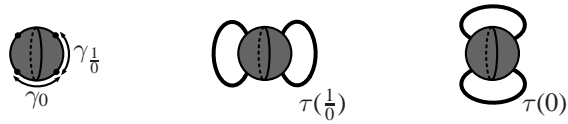}}
\end{center}
\caption{The arcs $\gamma_{\overzero}$ and $\gamma_0$ in the boundary of a tangle (left); and the `denominator' (centre) and `numerator' (right) closures denoted $\tau(\overzero)$ and $\tau(0)$ respectively when applied to a given tangle with fixed framing.}
\label{fig:closures}
\end{figure}

Set $M=\Br(B^3,\tau)$ and let $\alpha$ and $\beta$ be a pair of slopes in $\partial M$ that intersect geometrically in a single point. Then there is a choice of representative for the tangle $T$ compatible with the pair $\{\alpha, \beta\}$ in the sense that $\widetilde{\gamma_{\overzero}}=\alpha$ and $\widetilde{\gamma_0}=\beta$. Here, $\gamma_{\overzero}$ and $\gamma_{0}$ are arcs embedded in the boundary of the tangle with endpoints on $\partial \tau$ as in \fullref{fig:closures}, and $\{\widetilde{\gamma_{\overzero}},\widetilde{\gamma_0}\}$ is the pair of slopes in $\partial\left(\Br(B^3,\tau)\right)=\partial M$ covering $\gamma_{\overzero}$ and $\gamma_{0}$.

As a result, the Dehn fillings $M(\alpha)$ and $M(\beta)$ may be obtained as the two-fold branched covers $\Br(S^3,\tau(\overzero))$ and $\Br(S^3,\tau(0))$ respectively, where $\tau(\overzero)$ and $\tau(0)$ are the links resulting from the closures by the rational tangles shown in \fullref{fig:closures}.

In particular, $M(n\alpha+\beta)\cong\Br(S^3,\tau(n))$ in this notation, giving the branch sets for integer surgeries relative to the basis $\{\alpha,\beta\}$ (throughout we choose an orientation with $\alpha\cdot\beta=+1$). The closure giving rise to the branch set $\tau(n)$ is shown in \fullref{fig:fraction}. Notice that the half-twists in the base lift to full Dehn twists along $\alpha$ in the cover. More generally, we may write $M(p\alpha+q\beta)\cong\Br(S^3,\tau(r))$ where $r=\pq$ is given by the continued fraction expansion $[a_1,\ldots,a_\ell]$, and $\tau(r)$ is the link obtained by adding a rational tangle specified by the continued fraction to the tangle $T$. A specific example is shown in \fullref{fig:fraction}; for details see \cite[Section 3.3]{Watson2008}. 

\begin{figure}[ht!]
\begin{center}
\labellist \small 
	\pinlabel $\cdots$ at 382 464
	\pinlabel $\underbrace{\phantom{aaaaaaaaaaaa}}_n$ at 325 390
\endlabellist	
\raisebox{-18pt}{\includegraphics[scale=0.25]{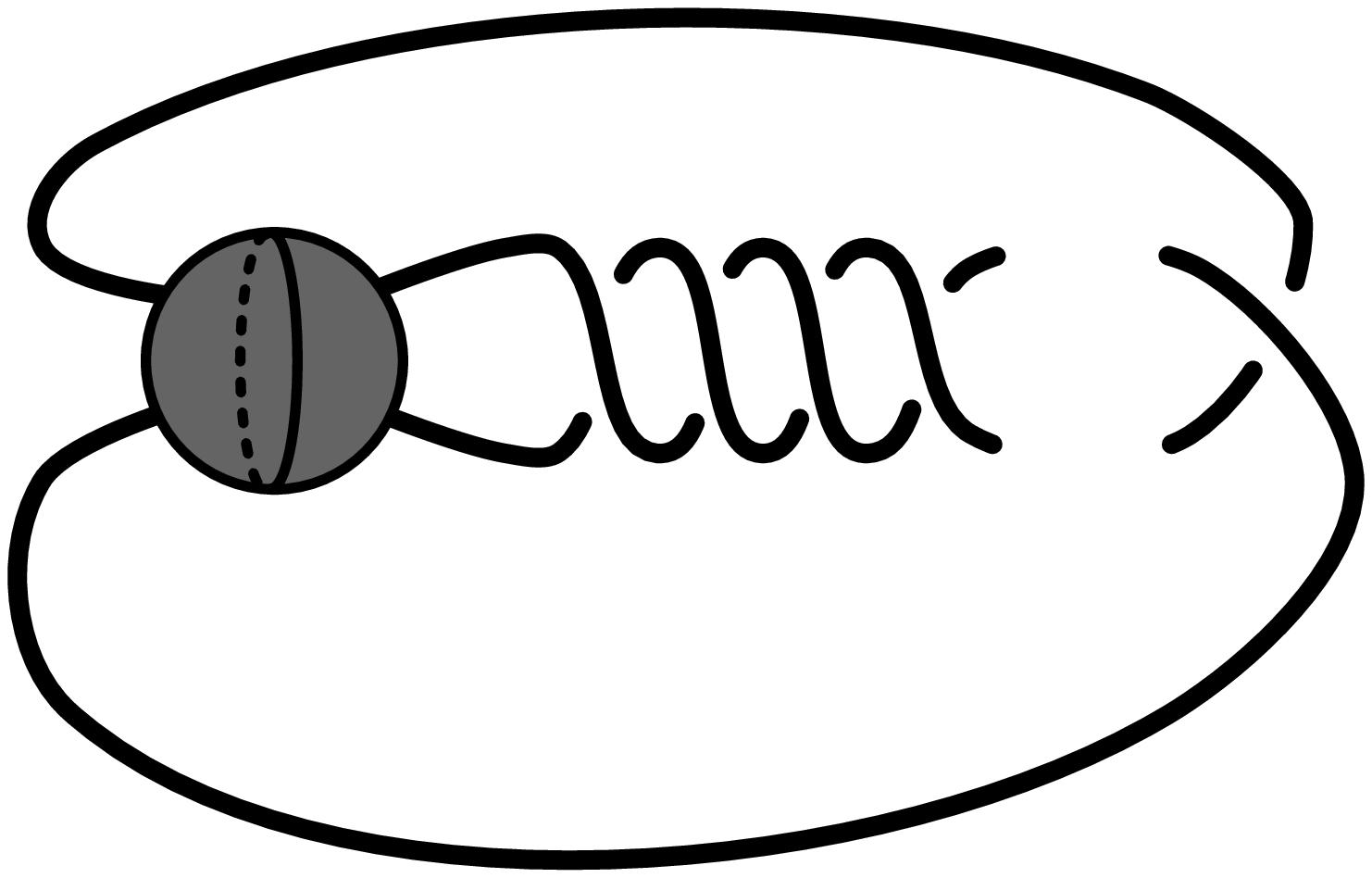}}\qquad\quad
\labellist \small 
	\pinlabel $\underbrace{\phantom{aaai}}_1\underbrace{\phantom{iiaaaaaaaaa}}_3\underbrace{\phantom{aaaaaaaa}}_3$ at 351 380
\endlabellist
\raisebox{0pt}{\includegraphics[scale=0.25]{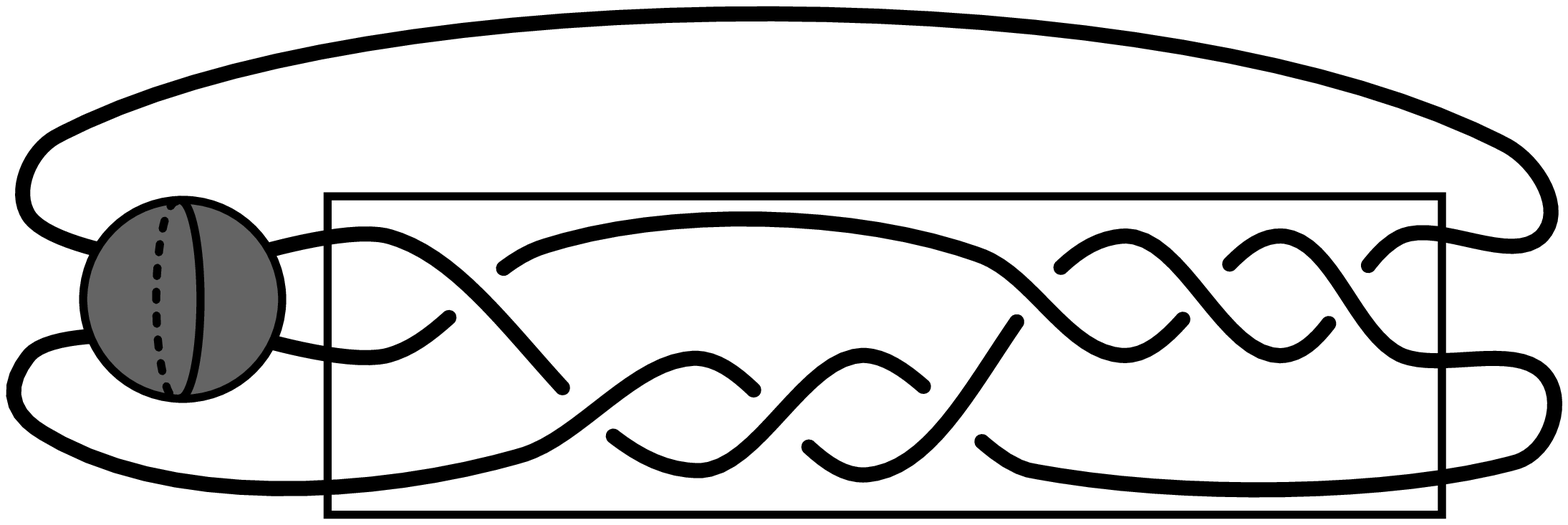}}
\end{center}
\caption{The closure $\tau(n)$ (left) giving rise to the branch set for integer surgeries (that is, Dehn fillings along slopes $n\alpha+\beta$), and the closure $\frac{13}{10}=[1,3,3]$ (right) corresponding to $13\alpha+10\beta$ Dehn filling in the cover.}
\label{fig:fraction}
\end{figure}

As a result, $\det(\tau(r))=c_M\Delta(p\alpha+q\beta,\lm)$ where $\Delta$ measures the minimal geometric intersection number between two slopes, $\lm$ is the rational longitude, and $c_M$ is some fixed constant depending only on $M$ \cite[Lemma 3.2]{Watson2008}. This is due to the fact that $\det(\tau(r))=|H_1(\Br(S^3,\tau(r));\bZ)|$ whenever $\Br(S^3,\tau(r))$ is a $\bQ$-homology sphere, and $\det(\tau(r))=0$ otherwise. Recall that the rational longitude is characterized by the property that some finite collection of like oriented copies of $\lm$ in $\partial M$ bound an essential surface in $M$. As a result, if $H_1(M;\bZ)\cong H\oplus\bZ$ for some finite abelian group $H$, then $c_M=\operatorname{ord}_{H}i_*(\lm)|H|$ where $i_*$ is the homomorphism induced by the inclusion $i\co\partial M\into M$. 

A particular useful class of examples is provided by strongly invertible knots. Recall that a knot $K$ is strongly invertible if there is an involution on the complement $S^3\smallsetminus\nu(K)$ that fixes a pair of arcs meeting the boundary torus transversally in exactly 4 points. It is an easy exercise to show that the quotient of such an involution is always a three-ball, and by keeping track of the image of the fixed point set in the quotient, $\tau$, such a knot complement is always the two-fold branched cover of a tangle $T=(B^3,\tau)$. In particular, we have the notation $S^3_{r}(K)=\Br(S^3,\tau(r))$ for $r$-surgery on the strongly invertible knot $K$, once a representative compatible with the preferred basis $\{\mu,\lambda\}$ (provided by the knot meridian $\mu$ and the Seifert longitude $\lambda$) has been fixed.

\section{Quasi-alternating tangles}\label{sec:QA-tangles}

\begin{definition} A quasi-alternating tangle $T=(B^3,\tau)$ admits a representative of the homeomorphism class of the pair with the property that $\tau(\overzero)$ and $\tau(0)$ are quasi-alternating links. Such a representative, when it exists, will be referred to as a quasi-alternating framing for $T$.\end{definition}

Note that $\tau(\overzero)$ and $\tau(0)$ are quasi-alternating if and only if the mirrors, $\tau(\overzero)^\star$ and $\tau(0)^\star$, are quasi-alternating. As a result the tangle $T$ is quasi-alternating if and only if $T^\star$ is a quasi-alternating tangle. This fact allows us to pass to the mirror and restrict to positive surgeries whenever convenient.   

\begin{theorem}\label{thm:main-qa} If $T$ is a quasi-alternating tangle then, up to taking mirrors and/or renaming $\gamma_{\overzero}$ and $\gamma_0$, the links $\tau(r)$ are quasi-alternating for all rational numbers $r\ge0$, where $T=(B^3,\tau)$ is a quasi-alternating framing. \end{theorem}

\begin{proof}[Sketch of proof]
This follows from \cite[Theorem 4.7 and Remark 4.5]{Watson2008}, but we sketch the argument here for the reader's convenience.

First notice that, up to taking mirrors, $\det(\tau(1))=\det(\tau(\overzero))+\det(\tau(0))$. This results from the observation that $\det(\tau(1))=c_M\Delta(\alpha+\beta,\lm)=c_m(\Delta(\alpha,\lm)+\Delta(\beta,\lm))$, where $M=\Br(B^3,\tau)$. As a result, $\tau(1)$ must be quasi-alternating. Next notice that, up to renaming $\gamma_{\overzero}$ and $\gamma_0$ (thereby renaming the pair of slopes $\alpha$ and $\beta$ in the cover), we have that $\det(\tau(n))=n\det(\tau(\overzero))+\det(\tau(0))$ for all $n\ge0$ by induction on $n$. 

The choices above ensure that we may orient $\partial M$ so that $\alpha\cdot\beta=+1$, $\alpha\cdot\lm>0$ and $\beta\cdot\lm>0$. With these choices fixed, the result follows from a second induction on $\ell$, the length of the continued fraction expansion of $r=[a_1,\ldots,a_\ell]$, applying the fact that $\det(\tau(r))=c_M\Delta(p\alpha+q\beta,\lm)$ where $r=\pq$. In particular, resolving the final crossing of the continued fraction gives $\det(\tau(r))=\det(\tau(r_0))+\det(\tau(r_1))$ as required, where $r_0$ and $r_1$ correspond to the $0$- and $1$-resolutions specified by the continued fractions $[a_1,\ldots,a_{\ell-1}]$ and $[a_1\ldots,a_{\ell}-1]$ (see \cite[Section 3.4]{Watson2008}).\end{proof} 

\begin{example} Any alternating projection of a non-split alternating link $L$ gives rise to a quasi-alternating tangle: it suffices to form a tangle by removing two arcs of $L$ in such a way that the crossings adjacent to the endpoints of the tangle alternate as they are traversed in order around the diagram of the tangle. See \fullref{fig:trefoil-and-tangle}, for example. As a result, any alternating link gives rise to an infinite family of quasi-alternating links by applying \fullref{thm:main-qa}.\end{example}

\begin{figure}[ht!]
\begin{center}
\labellist \small
	\pinlabel $o$ at 518 478
	\pinlabel $u$ at 459 474
	\pinlabel $o$ at 460 287
	\pinlabel $u$ at 517 282
\endlabellist
\raisebox{0pt}{\includegraphics[scale=0.25]{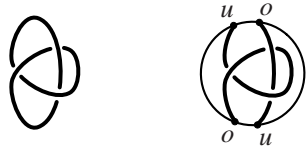}}
\end{center}
\caption{An alternating projection of the trefoil (left) gives rise to a quasi-alternating tangle (right). Notice that both $\tau(\overzero)$ and $\tau(0)$ are alternating links.}
\label{fig:trefoil-and-tangle}
\end{figure}

For the particular instance of this example shown in \fullref{fig:trefoil-and-tangle}, the resulting quasi-alternating tangle is a rational tangle. As a result, the two-fold branched cover of the tangle is a solid torus, and this gives an alternative view on why this particular tangle is quasi-alternating: since every closure of this tangle corresponds to a surgery on the trivial knot in the two-fold branched cover, it follows from work of Hodgson and Rubinstein  that the resulting branch sets are always two-bridge links \cite{HR1985}. These are non-split alternating links -- and hence quasi-alternating -- for all but the zero surgery.

\fullref{fig:two-framings} shows second quasi-alternating tangle arising in this way that is not a rational tangle.

\begin{figure}[ht!]
\begin{center}
\raisebox{0pt}{\includegraphics[scale=0.25]{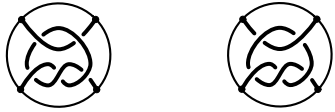}}
\end{center}
\caption{Two quasi-alternating framings for a given tangle: notice that these tangles differ by a homeomorphism that does not fix the boundary. This tangle is two-fold branched covered by the complement of the right-hand trefoil (see \fullref{sec:lens} and, more generally, \fullref{sec:QA-Montesinos}) as demonstrated in \cite[Figure 5]{Watson2008}, for example.}
\label{fig:two-framings}
\end{figure}

Further examples arise naturally in the context of Dehn surgery on a strongly invertible knots more generally, and this is explored in the following sections.  

\section{Lens space surgeries and quasi-alternating links}\label{sec:lens}

Moser showed that the positive $(p,q)$-torus knot yields a lens space for $(pq-1)$-surgery \cite{Moser1971}. More generally, Berge gives a conjecturally complete list of knots admitting lens space surgeries that includes torus knots \cite{Berge}, and results due to Schreier \cite{Schreier1924} and Osborne \cite{Osborne1981} show that every such knot is strongly invertible by virtue of having Heegaard genus at most 2. That is, up to mirrors, there is a positive integer $N$, depending on $K$, for which $S^3_N(K)$ is a lens space whenever $K$ is a knot in Berge's list. 

As a result, the quotient tangle $T=(B^3,\tau)$ compatible with $\{\mu,N\mu+\lambda\}$ is a framed quasi-alternating tangle: $\tau(\overzero)$ is unknotted while $\tau(0)$ must be a non-split two-bridge link due to work of Hodgson and Rubinstein \cite{HR1985}. Therefore every Berge knot complement may be viewed as the two-fold branched cover of a quasi-alternating tangle, and in summary we have the following immediate application of \fullref{thm:main-qa} (c.f. \cite[Proposition 4.8]{Watson2008}):

\begin{proposition}\label{prp:berge} Up to mirrors, every Berge knot has property {\bf QA}. \end{proposition}

\section{Pretzels admitting L-space surgeries}

Let $P_q$ denote the $(-2,3,q)$-pretzel (for odd positive integers $q$); $P_7$ is illustrated in \fullref{fig:237-pretzel}. It is known that each $P_q$ admits L-space surgeries \cite{OSz2005-lens}. Indeed, $P_1$, $P_3$ and $P_5$ are the $(5,2)$-, $(3,4)$- and $(3,5)$- torus knots, respectively, and $P_7$ is a non-torus Berge knot (it was demonstrated by Fintushel and Stern that this hyperbolic knot admits lens space surgeries \cite{FS1980}). In general $P_q$ is hyperbolic when $q>5$, however $P_q$ does not admit finite fillings (in particular, does not admit lens space surgeries) for $q>9$ \cite{Mattman2000}.

\begin{figure}[ht!]
\begin{center}
\raisebox{0pt}{\includegraphics[scale=0.25]{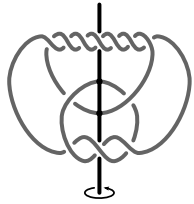}}
\end{center}
\caption{The $(-2,3,7)$-pretzel knot $P_7$ with axis for the strong inversion labelled.}
\label{fig:237-pretzel}\end{figure}

Since the knot $P_q$ is strongly invertible, surgeries on $P_q$ are two-fold branched covers of $S^3$. Moreover, we have the following:

\begin{theorem}\label{thm:pretzels} The knot $P_q$ has property {\bf QA} for ever odd positive integer $q$.\end{theorem}

\begin{proof} Since $P_1$ is a torus knot, the case $q=1$ (and, indeed, the cases $q=3,5,7$) follows from \fullref{prp:berge}. To prove the theorem, assume without loss of generality that $q>1$. We first need to construct the associated quotient tangle for each $P_q$. This amounts to choosing a fundamental domain for the action of the involution on the knot complement (illustrated in \fullref{fig:237-pretzel}), and is done in \fullref{fig:23q-tangle}. 

\begin{figure}[ht!]
\begin{center}
\labellist \small
	\pinlabel $\frac{q-1}{2}$ at 234 537
\endlabellist
\raisebox{0pt}{\includegraphics[scale=0.25]{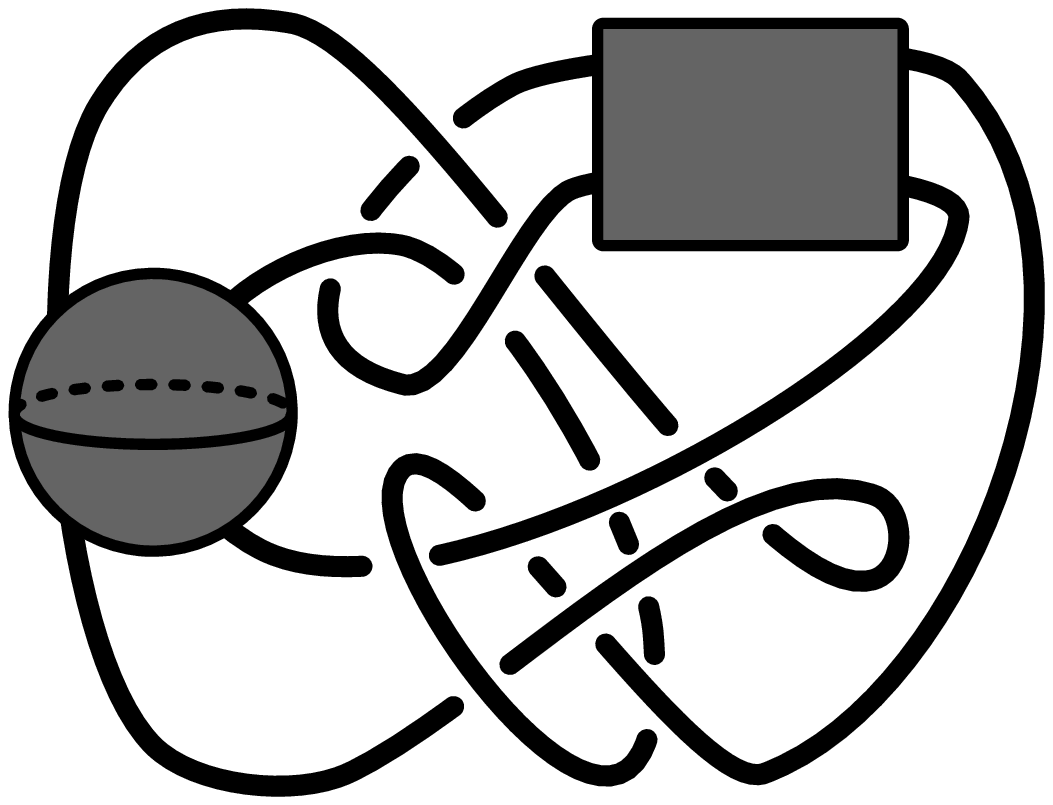}}\qquad
\labellist \small
	\pinlabel $\frac{q-1}{2}$ at 357 542
\endlabellist
\raisebox{0pt}{\includegraphics[scale=0.25]{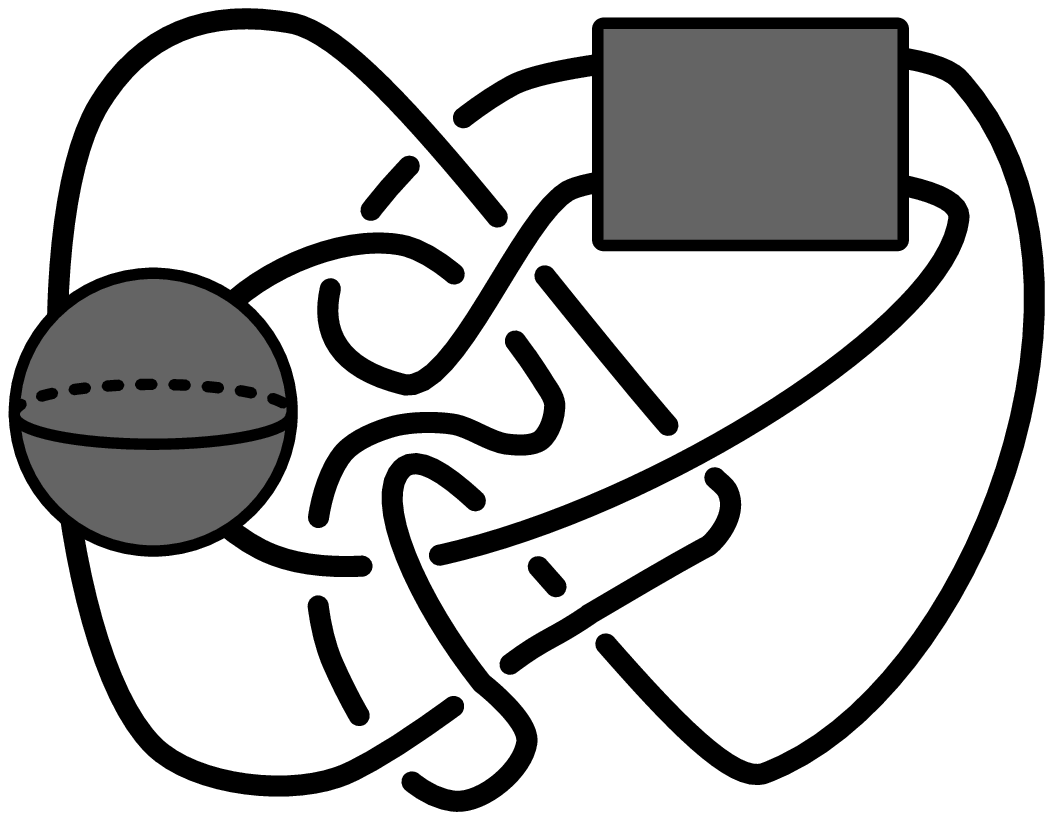}}\qquad
\labellist \small
	\pinlabel $\frac{q-3}{2}$ at 315 515
\endlabellist
\raisebox{0pt}{\includegraphics[scale=0.25]{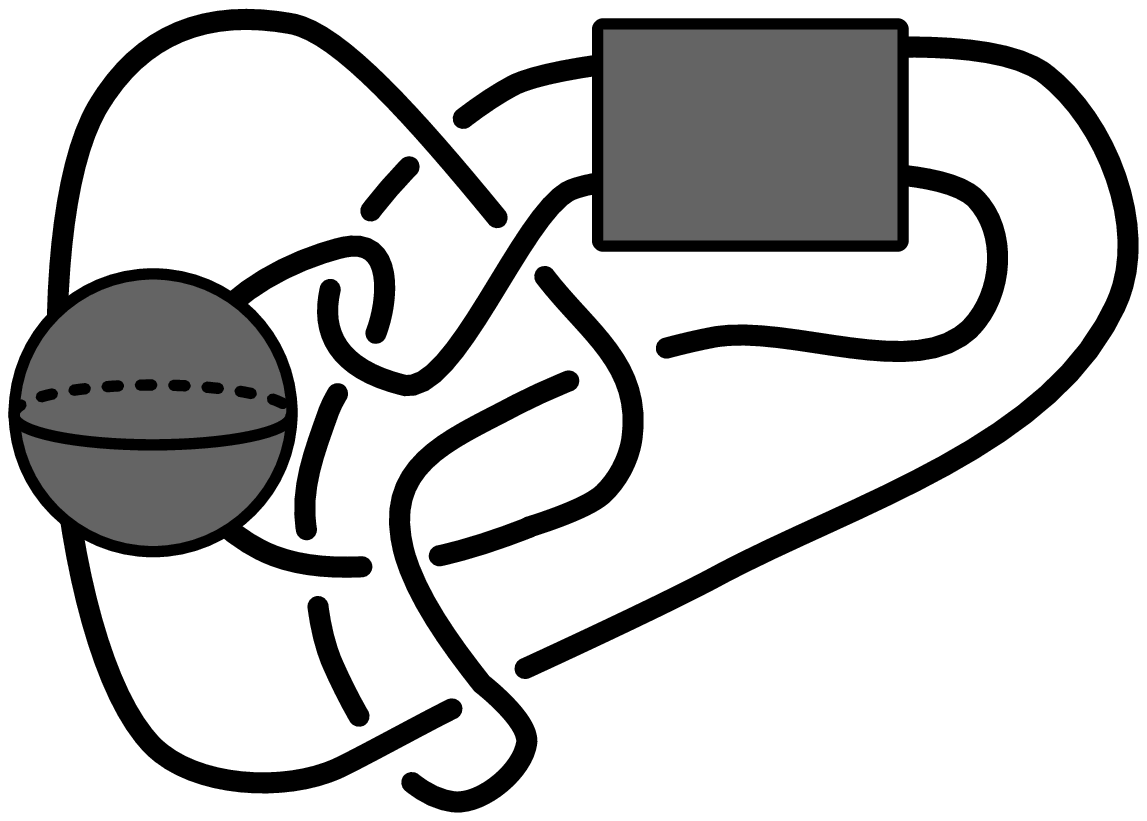}}
\end{center}
\caption{Simplifying the tangle resulting from the quotient of $P_q$ by the strong inversion. Notice that the last step involves a homeomorphism that does not fix the boundary sphere. }
\label{fig:23q-tangle}\end{figure}

The preferred framing for this tangle, which we denote by $T_q=(B^3,\tau_q)$, is shown in \fullref{fig:23q-framed}. To see this, it suffices to check the image of the longitude in the quotient of $P_3$, then observe that the framing changes by 4 half-twists in the quotient when $q$ is replaced by $q+2$.  In particular, notice that the branch set associated to $S^3_{2q+5}(P_q)\cong\Br(S^3,\tau_q(2q+5))$ has a fixed number of twists corresponding to the framing, so that only the twists of the second box vary in $q$ for this family of branch sets.  

\begin{figure}[ht!]
\begin{center}
\labellist 
	\small 
	\pinlabel $\frac{q-3}{2}$ at 374 435
	\scriptsize
	\pinlabel $-2q\!-\!7$ at 171 435
\endlabellist
\raisebox{0pt}{\includegraphics[scale=0.25]{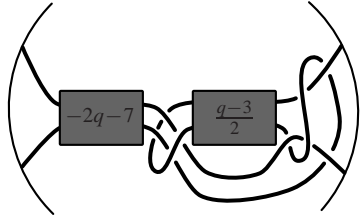}}\qquad
\end{center}
\caption{The preferred representative $T_q=(B^3,\tau_q)$ of the associated quotient tangle for the knot $P_q$. Note that the twists in the left-most box correspond to the framing, while adding a positive half-twist in the right-most box corresponds to exchanging the knot $P_q$ for $P_{q+2}$ in the cover.}
\label{fig:23q-framed}\end{figure}

We claim that the branch sets $\tau_q(2q+5)$ are quasi-alternating. To see this, consider the tangle $T'=(B^3,\tau')$ shown in \fullref{fig:23q-qa}. This tangle has the property, by construction, that $\tau'(n)=\tau_{2n+3}(4n+11)=\tau_q(2q+5)$, and as such $\det(\tau'(n))=4n+11$.

\begin{figure}[ht!]
\begin{center}
\raisebox{0pt}{\includegraphics[scale=0.25]{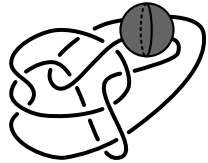}}
\end{center}
\caption{The quasi-alternating tangle $T'=(B^3,\tau')$ for which $\tau'(n)=\tau_{2n+3}(4n+11)=\tau_q(2q+5)$.}
\label{fig:23q-qa}\end{figure}

Moreover, as shown in \fullref{fig:412} the branch set $\tau'(\overzero)$ is the $(2,4)$-torus link hence $\det(\tau(\overzero))=4$. As a result, $\det(\tau(n))=n\det(\tau(\overzero))+11$. Notice that, as an alternating link, $\tau'(\overzero)$ is quasi-alternating. 

\begin{figure}[ht!]
\begin{center}
\raisebox{0pt}{\includegraphics[scale=0.25]{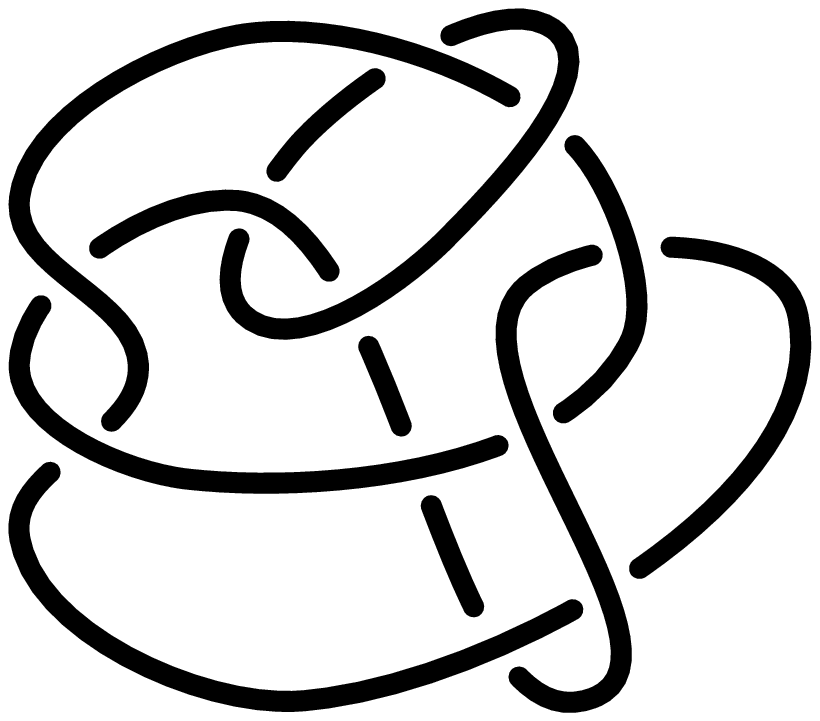}}\qquad
\raisebox{0pt}{\includegraphics[scale=0.25]{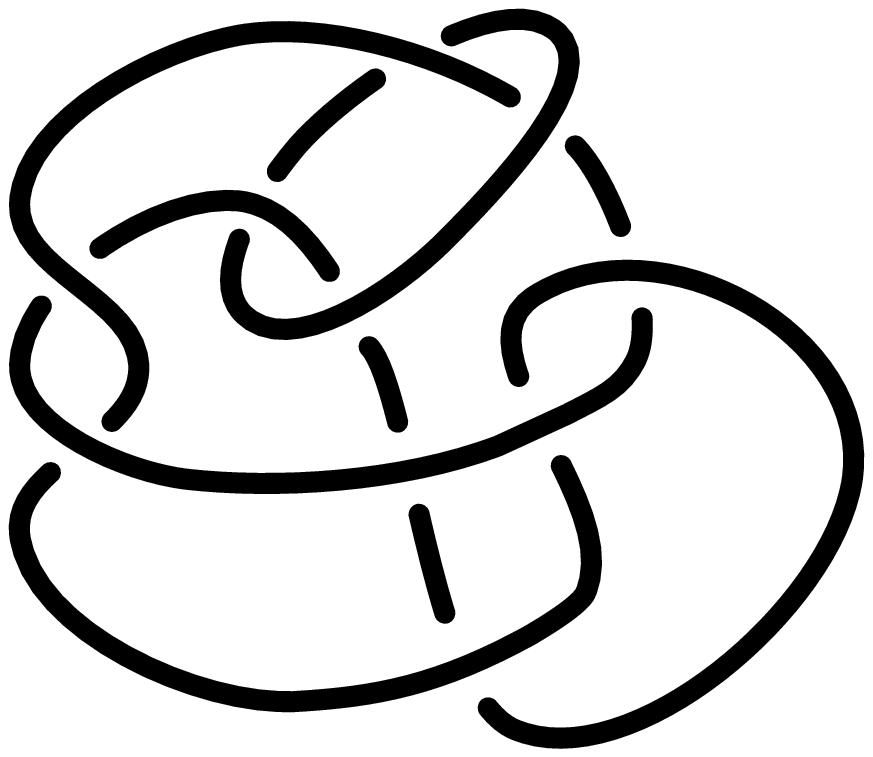}}\qquad
\raisebox{0pt}{\includegraphics[scale=0.25]{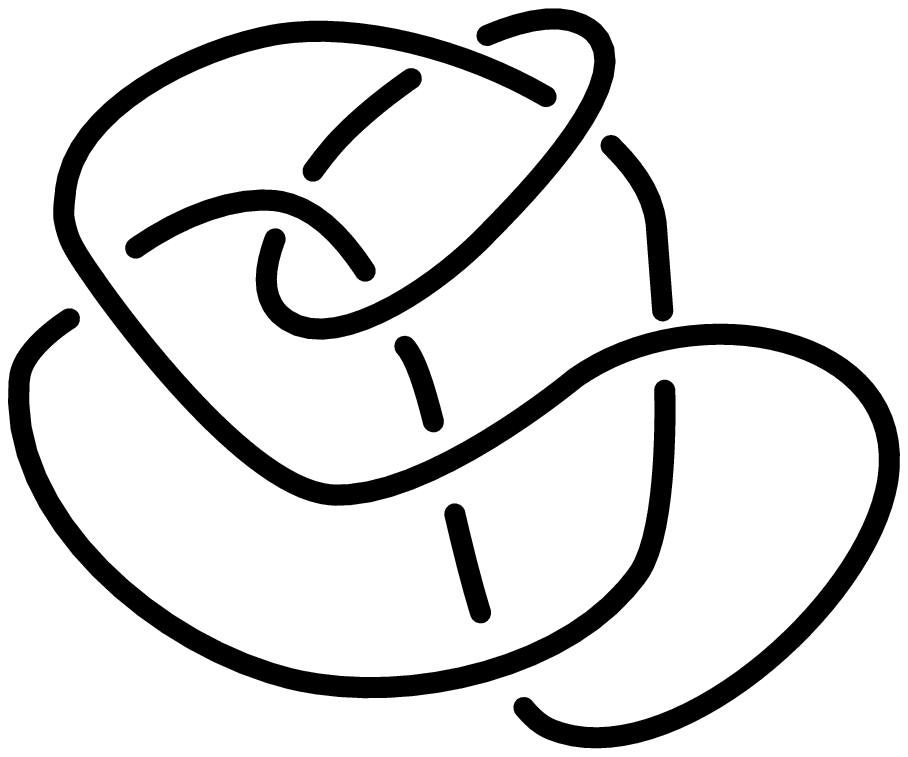}}\qquad
\raisebox{10pt}{\includegraphics[scale=0.25]{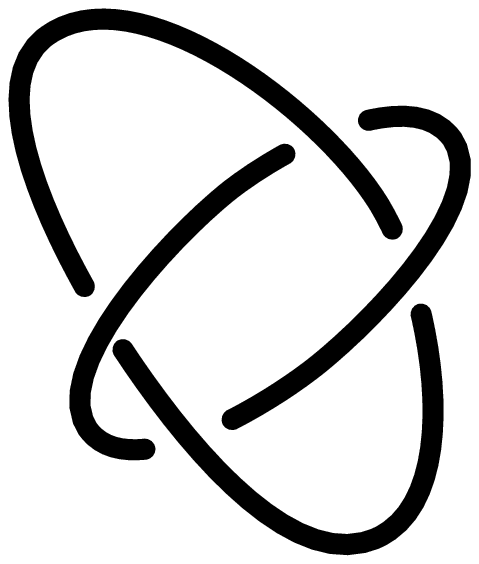}}
\end{center}
\caption{Isotopy of the link $\tau(\overzero)$ yields the $(2,4)$-torus link.}
\label{fig:412}\end{figure}

Therefore, if we can demonstrate that $\tau'(0)$ is quasi-alternating with $\det(\tau(0))=11$, we can conclude that $\tau'(n)$ is quasi-alternating for all $n\ge0$ (indeed, that $T'=(B^3,\tau')$ is a framed quasi-alternating tangle).   This is indeed the case: the isotopy in \fullref{fig:7_2} yields an alternating diagram for $\tau(0)\simeq 7_2$ and it is well known that this two-bridge knot has $\det(\tau(0))=11$. 

\begin{figure}[ht!]
\begin{center}
\raisebox{0pt}{\includegraphics[scale=0.25]{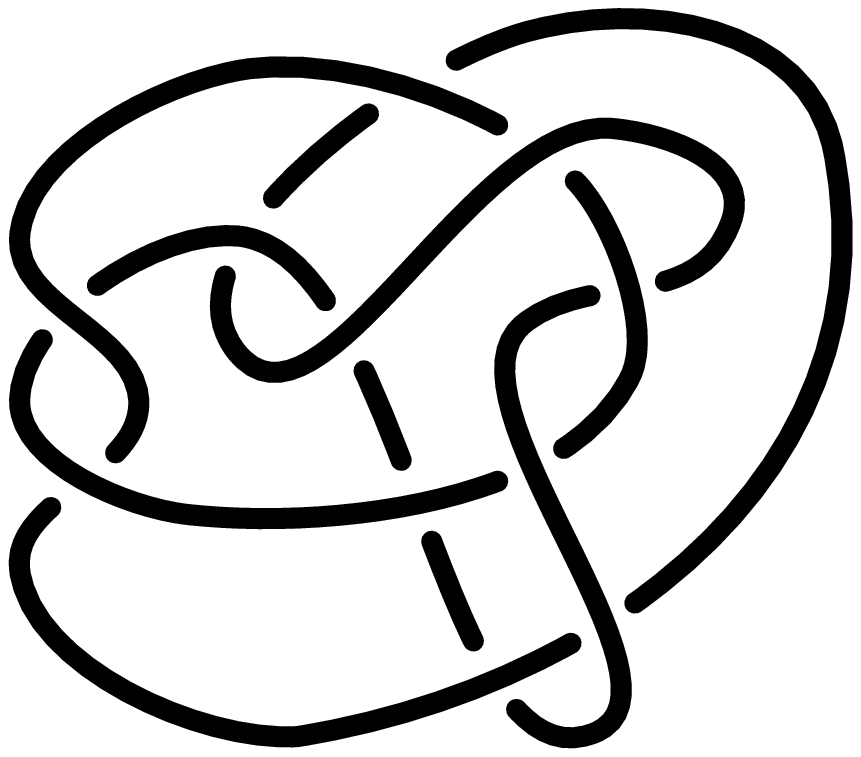}}\ \
\raisebox{-5pt}{\includegraphics[scale=0.25]{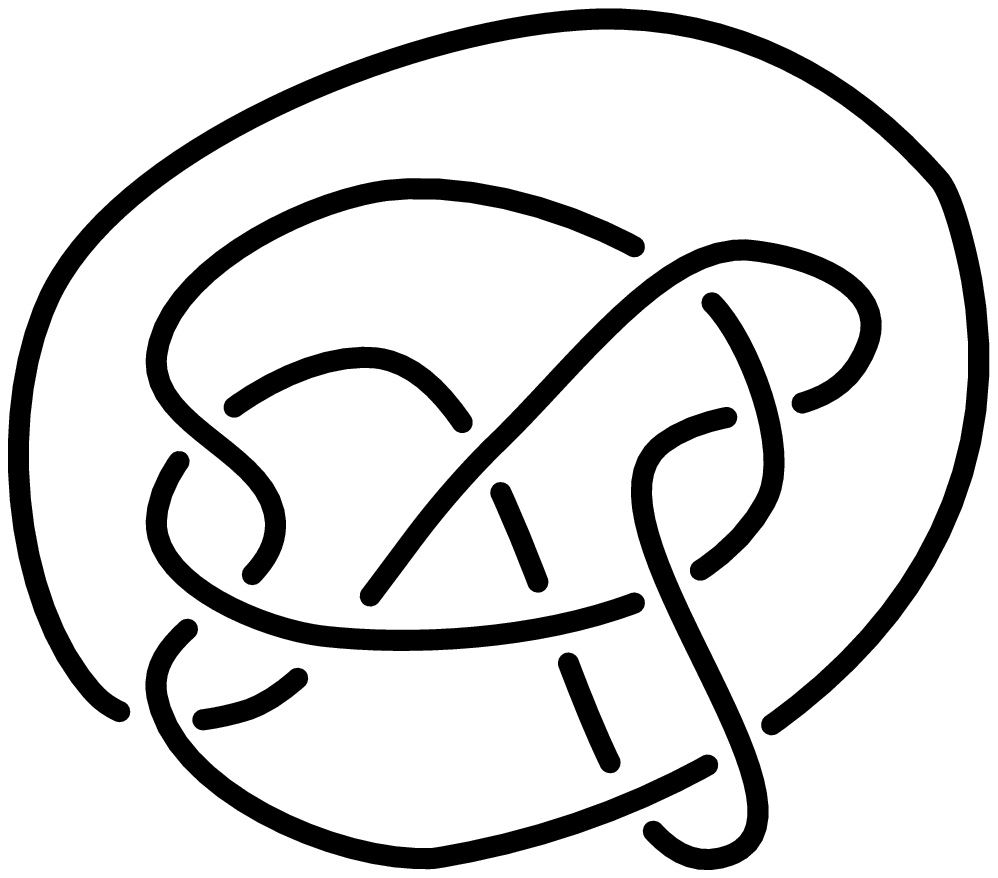}}\ \
\raisebox{-10pt}{\includegraphics[scale=0.25]{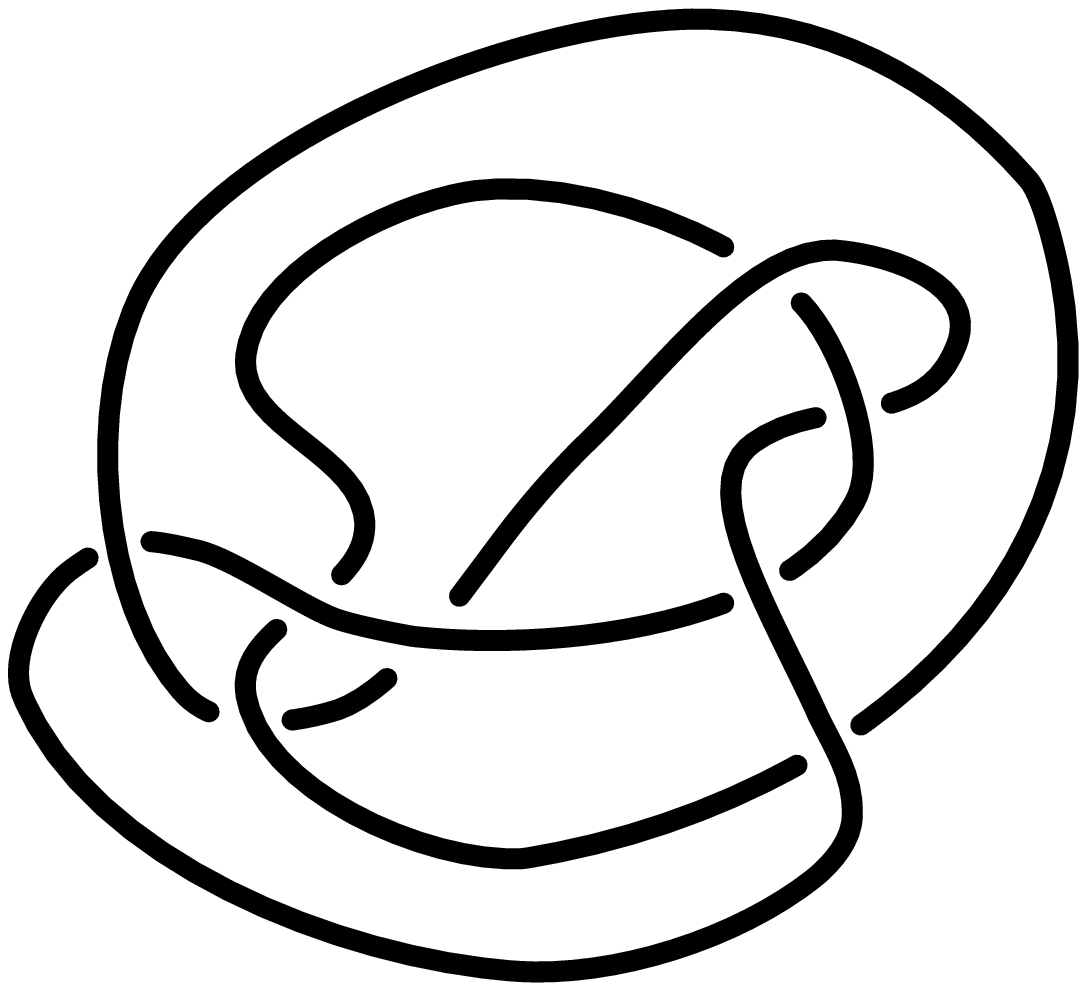}}\ \
\raisebox{3pt}{\includegraphics[scale=0.25]{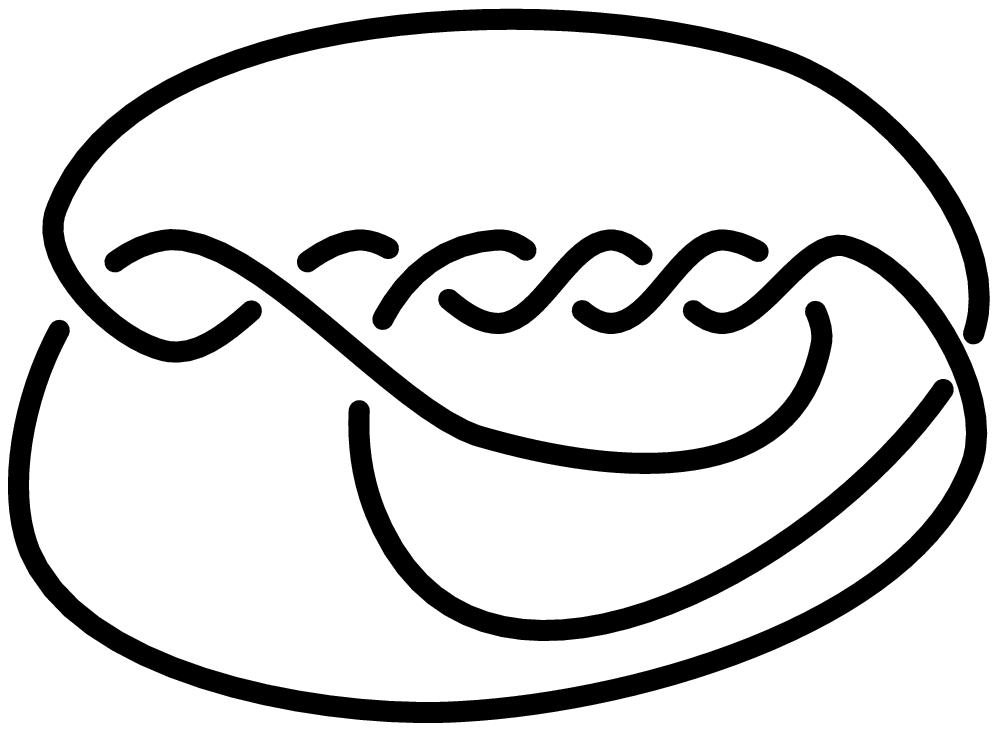}}\ \
\raisebox{5pt}{\includegraphics[scale=0.25]{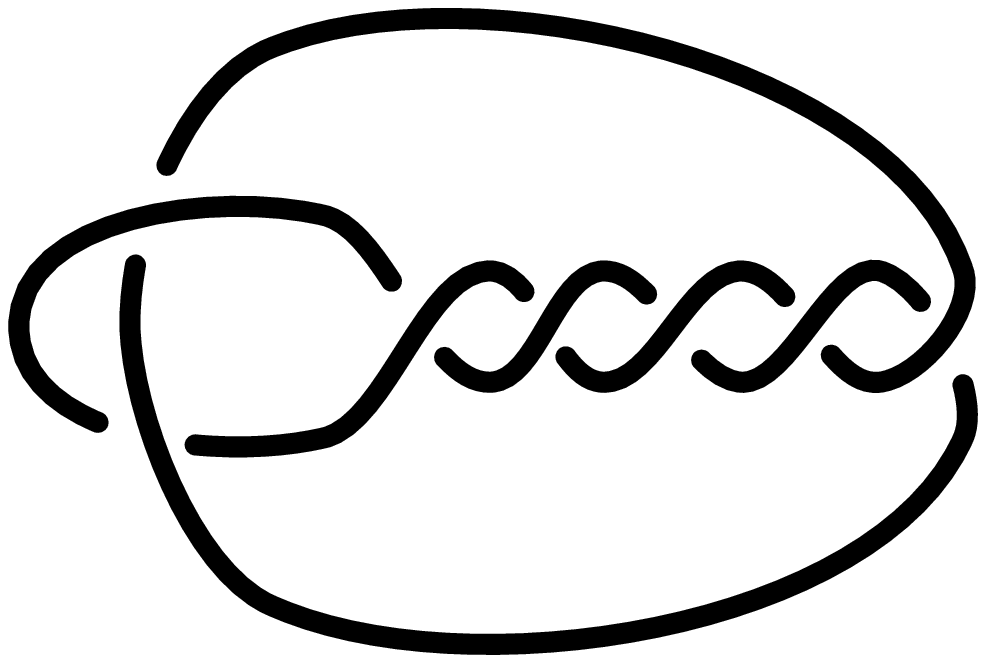}}
\end{center}
\caption{Isotopy of the knot $\tau(0)$ yields the knot $7_2$.}
\label{fig:7_2}\end{figure}

To complete the proof notice that we have shown, for each $q$, that there is a positive integer surgery for which the corresponding branch set is quasi-alternating. In particular, by fixing a representative for $T_q$ compatible with $\{\mu,(2q+5)\mu+\lambda\}$ we obtain a framed quasi-alternating tangle. Therefore, the manifold $S^3_{r}(P_q)$ is the two-fold branched cover of a quasi-alternating link for all rational numbers $r\ge 2q+5$ by applying \fullref{thm:main-qa}.  \end{proof}

We point out an interesting consequence of the construction used in the proof of \fullref{thm:pretzels}: there are composite, quasi-alternating knots realized as quasi-alternating on a non-composite diagram. Indeed, since $P_5$ is the $(3,5)$-torus knot, the manifold $S^3_{15}(P_5)$ is a connect sum of lens spaces \cite{Moser1971}. Therefore, the associated quasi-alternating branch set $\tau_5(15)=\tau'(1)$ must be a connect sum of two-bridge knots \cite{HR1985}. We have demonstrated that $\tau'(n)$ is quasi-alternating for all $n$, and leave it as a challenge to the interested reader to demonstrate that $\tau'(1)$ is the knot $3_1\# 4_1$. Notice that since both $\tau'(\overzero)$ and $\tau'(0)$ are non-trivial prime links, the resolved crossing must constitute an interaction between the connect summands of $\tau'(1)$. We point out that it is an open problem to show that $K_1\# K_2$ quasi-alternating implies that the $K_i$ are quasi-alternating, and this example suggests that an answer to this question might be subtle.  

\section{On Cabling}

Let $\sC_{q,p}(K)$ denote the $(q,p)$-cable of a knot $K$. By this convention, {\em sufficiently positive cabling} refers to taking sufficiently large $p$.

\begin{theorem}\label{thm:cables} Let $K$ be a knot with property {\bf QA}. Then all sufficiently positive cables of $K$ have property {\bf QA} as well. \end{theorem}

 \begin{proof}
We begin by recalling a result due to Gordon \cite{Gordon1983}.
Let $r=\frac{kpq\pm1}{k}$. Then since the cable space corresponding to $\sC_{q,p}$ is Seifert fibered over an annulus with a single cone point of order $q$, $r$-surgery on $\sC_{q,p}(K)$ may be obtained via surgery on the original knot $K$: $S^3_{r}(\sC_{q,p}(K))\cong S^3_{r/q^2}(K)$ (see \cite[Corollary 7.3]{Gordon1983}). In particular, we have that $S^3_{pq-1}(\sC_{q,p}(K))\cong S^3_{(pq-1)/q^2}(K)$ so that $\frac{pq-1}{q^2}$ is an increasing function in $p$.

Now suppose that $K$ has property \QA. Then for $p\gg0$ we can be sure that $S^3_{(pq-1)/q^2}(K)$ is the two-fold branched cover of a quasi-alternating link. Therefore, $S^3_{pq-1}(\sC_{q,p}(K))$ must be a two-fold branched cover of a quasi-alternating link as well. Now since the cable of a strongly invertible knot is strongly invertible, the representative of the associated quotient tangle to $\sC_{q,p}(K)$ compatible with $(\mu,(pq-1)\mu+\lambda)$ is a quasi-alternating tangle for all $p$ sufficiently large. This observation, together with an application of \fullref{thm:main-qa}, proves the claim.  
\end{proof}
Notice that the proof of \fullref{thm:cables} depends only on properties of Dehn surgery and makes no reference to diagrams of the branch sets. Moreover, we recover a result due to Hedden that sufficiently positive cables of L-space knots provide new examples of L-space knots \cite[Theorem 1.10]{Hedden2008}. However, Hedden demonstrates that, given an L-space knot $K$, $\sC_{q,p}(K)$ admits L-space surgeries whenever $p\ge q(2g(K)-1)$ where $g(K)$ denotes the Seifert genus of $K$. It is natural to ask if $\sC_{q,q(2g(K)-1)}(K)$ has property \QA, as these seem to be possible candidates for L-space knots failing property \QA. 

\section{Constructing quasi-alternating Montesinos links.}\label{sec:QA-Montesinos}

We conclude by observing that the interaction between Dehn surgery (in the cover) and quasi-alternating tangles (in the base) provides a recipe for constructing infinite families of quasi-alternating Montesinos links. As shown below, every Montesinos link arises in this way, and as such it would be interesting to compare this approach with the constructions of Champanerkar and Kofman \cite{CK2009} and Widmer \cite{Widmer2008}. We remark that there is a complete classification Seifert fibered L-spaces (with base orbifold $S^2$, a class to which we restrict for the remainder of this work) due to \OSz\ \cite{OSz2003-plumbed}, and the construction given here shows that many of these are realized as the two-fold branched cover of a quasi-alternating link. For background on Dehn filling Seifert fibered manifolds we refer the reader to Boyer \cite{Boyer2002}. 

Let $M$ be a Seifert fibered space with base orbifold a disk with $n$ cone points denoted $D^2(p_1,\ldots,p_n)$. Then $\partial M$ is a torus, with a distinguished slope $\fibre$ given by a regular fibre in the boundary. The following is due to Heil \cite{Heil1974}:
\begin{theorem}\label{thm:heil}Given an $M$ as described above, $M(\fibre)$ is a connect sum of $n$ lens spaces, and for any slope $\alpha\ne\fibre$ we have that the Dehn filling $M(\alpha)$ is Seifert fibered with base orbifold $S^2(p_1,\ldots,p_n,\Delta(\alpha,\fibre))$.\end{theorem} 

This result generalizes Moser's results \cite{Moser1971} pertaining to torus knots (see \fullref{sec:lens}), since a $(p,q)$-torus knot is a regular fibre in a Seifert fibration of $S^3$ with base orbifold $S^2(p,q)$. Note in particular that in the present setting $M(\fibre)$ necessarily a two-fold branched cover of $S^3$ with branch set given by a connect sum of $n$ two-bridge links \cite{HR1985}. More generally, a result of Montesinos says that such manifolds may be obtained as the two-fold branched cover of a tangle $M=\Br(B^3,\tau)$ (though in this setting it may be that the fixed point set $\tau$ includes some closed components in addition to the pair of arcs), so that $M(\alpha)\cong\Br(S^3,L)$ where $L$ is a Montesinos link composed of $n$ rational tangles whenever $\alpha\ne\fibre$ \cite{Montesinos1976}.

Our aim is to identify when $L$ is quasi-alternating, by implicitly extending property \QA\ to regular fibres in certain Seifert fibered spaces. 

To this end, fix a regular fibre $\fibre$ in the boundary of $M$, and denote by $\lm$ the rational longitude. We will suppose that $\fibre\ne\lm$; this assumption ensures that $M$ is not the twisted $I$-bundle over the Klein bottle (for this case we refer the reader to \cite[Section 5.2]{Watson2008}). Given an orientation of $\lm$, fix an orientation on $\fibre$ so that $\fibre\cdot\lm>0$. Now for any slope $\mu$ with the property that $\mu\cdot\fibre=+1$ we have that $\alpha=\mu+k\fibre$ shares this same property for every integer $k$. As a result, $\alpha\cdot\lm>0$ for $k\gg0$. Fix such a $k$, together with a representative for the associated quotient tangle $T$ compatible with the resulting pair $\{\alpha,\fibre\}$. Then by \fullref{thm:heil} (together with the preceding discussion), the branch set for the fibre filling $\tau(0)$ is a connect sum of $n$ two-bridge links and $\tau(\overzero)$ is a Montesinos link composed of $n$ rational tangles. 

\begin{proposition}\label{prp:qa-seifert}
With the above notation, if $\tau(\overzero)$ is a quasi-alternating link, then $T$ is a quasi-alternating tangle. In particular, $\tau(r)$ is a quasi-alternating link for all $r\ge0$.
\end{proposition}
\begin{proof}This is immediate from the set up preceding the statement of the proposition, together with an application of \fullref{thm:main-qa}.\end{proof}

This proposition may be used to construct infinite families of quasi-alternating links, in an obvious manner, provided the quasi-alternating requirement on $\tau(\overzero)$ may be established. 

\subsection{Small Seifert fibered spaces} A Seifert fibered space is called {\em small} if it contains exactly 3 exceptional fibres (that is, has base orbifold $S^2(p_1,p_2,p_3)$). Notice that, by applying \fullref{thm:heil}, such manifolds may be obtained by Dehn filling 
of some Seifert fibered $M$ with base orbifold $D^2(p_1,p_2)$. Moreover, there is a slope $\alpha$ with $\alpha\cdot\lm>0$ and $\alpha\cdot\fibre=+1$ so that $M(\alpha)$ is a lens space (with base orbifold $S^2(p_1,p_2)$). Let $T$ be the associated quotient tangle so that $M\cong\Br(B^3,\tau)$, with representative chosen compatible with $\{\alpha,\fibre\}$. As a particular instance of \fullref{prp:qa-seifert}, we have demonstrated the following (compare \fullref{prp:berge} in the case of torus knot exteriors):

\begin{proposition}\label{prp:qa-small} Every $M$ as described above is the two-fold branched cover of a quasi-alternating tangle.\end{proposition} 

\subsection{An iterative construction}

\fullref{prp:qa-small} establishes that \fullref{prp:qa-seifert} is not vacuous. In particular, we observe that a regular fibre in a Seifert fibration of a lens space is a knot admitting a lens space surgery, and indeed provides an example of a knot satisfying a more general form of property \QA\ up to mirrors.

With this as a base case, it is now clear that if $Y$ is a Seifert fibration with base orbifold $S^2(p_1,\ldots,p_n)$, and $Y$ is the two-fold branched cover of a quasi-alternating link, then a regular fibre in $Y$ is a knot satisfying a more general form of property $\QA$. We end with an example demonstrating how this fact may be used to generate infinite families of quasi-alternating Montesinos links, using Dehn surgery to control the construction.  

 \begin{figure}[ht!]
\begin{center}
\labellist \small
	\pinlabel $\alpha$ at 323 485
	\pinlabel $\fibre$ at 178 342
\endlabellist
\raisebox{0pt}{\includegraphics[scale=0.25]{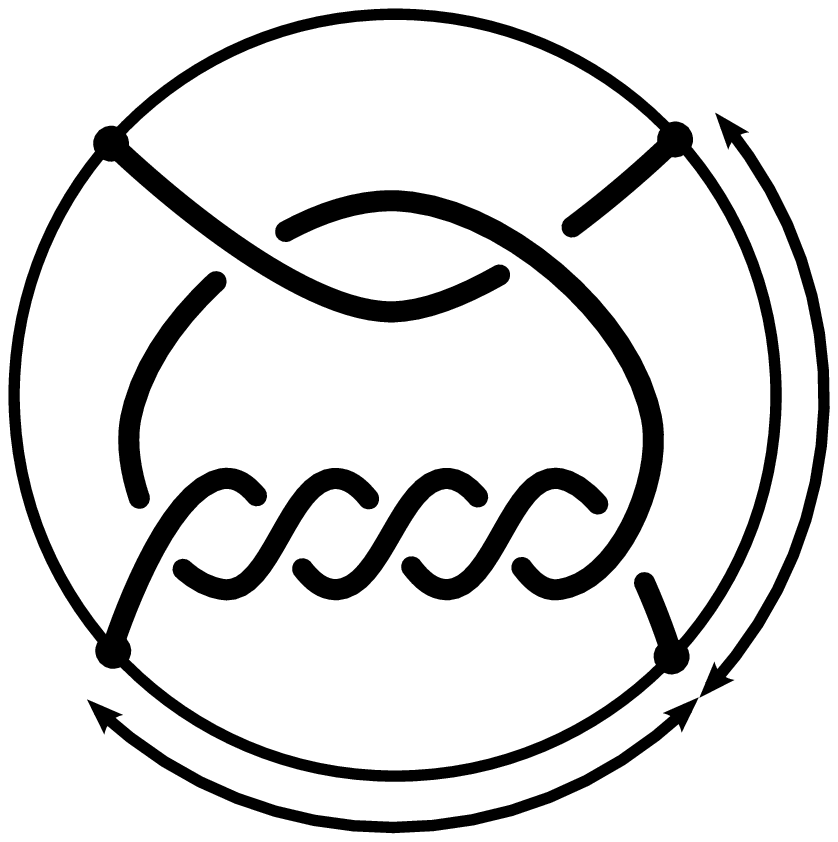}}\qquad\qquad
\raisebox{-10pt}{\includegraphics[scale=0.25]{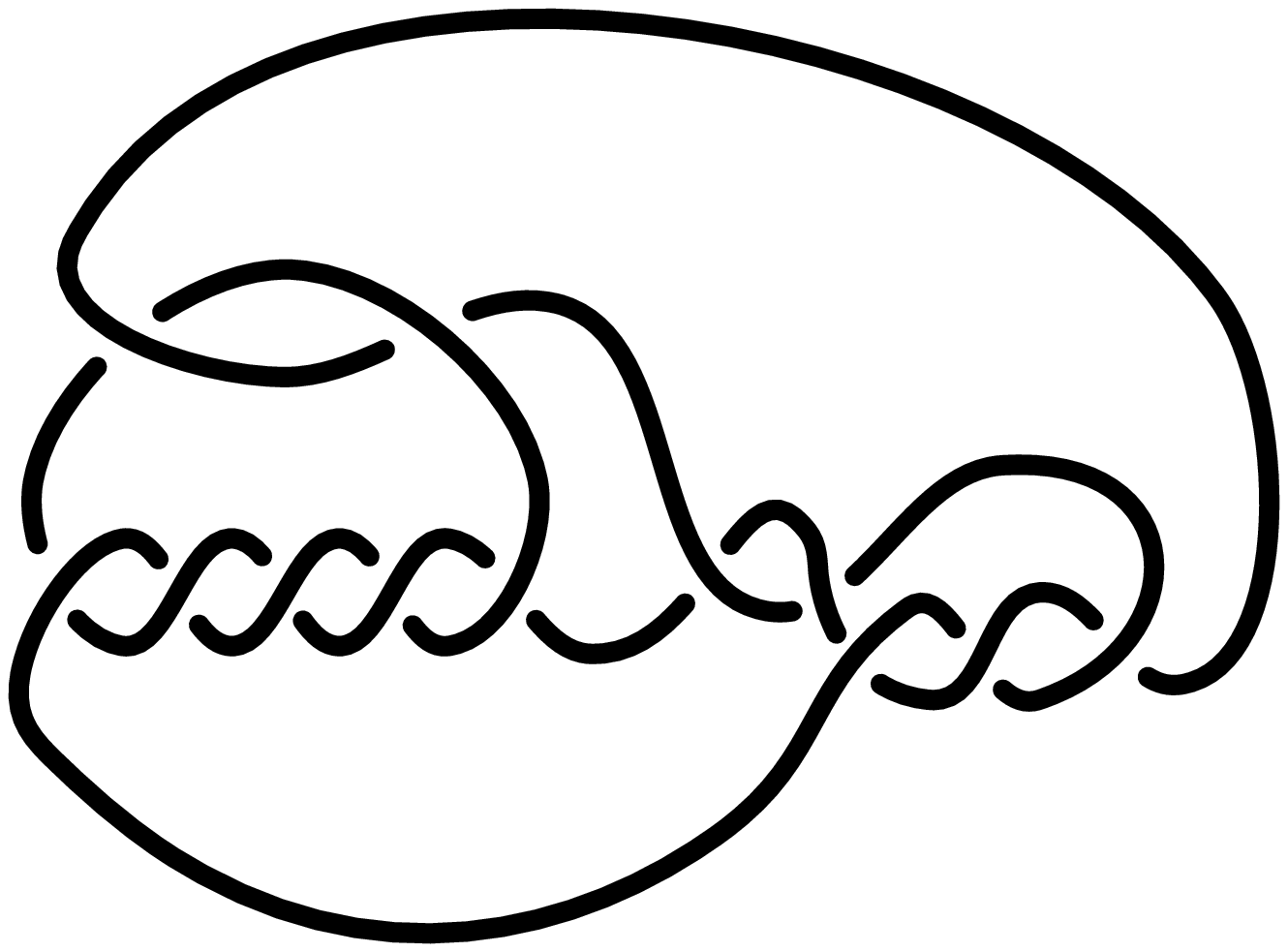}}\qquad\qquad
\labellist \small
	\pinlabel $\gamma$ at 203 652
\endlabellist
\raisebox{-5pt}{\includegraphics[scale=0.25]{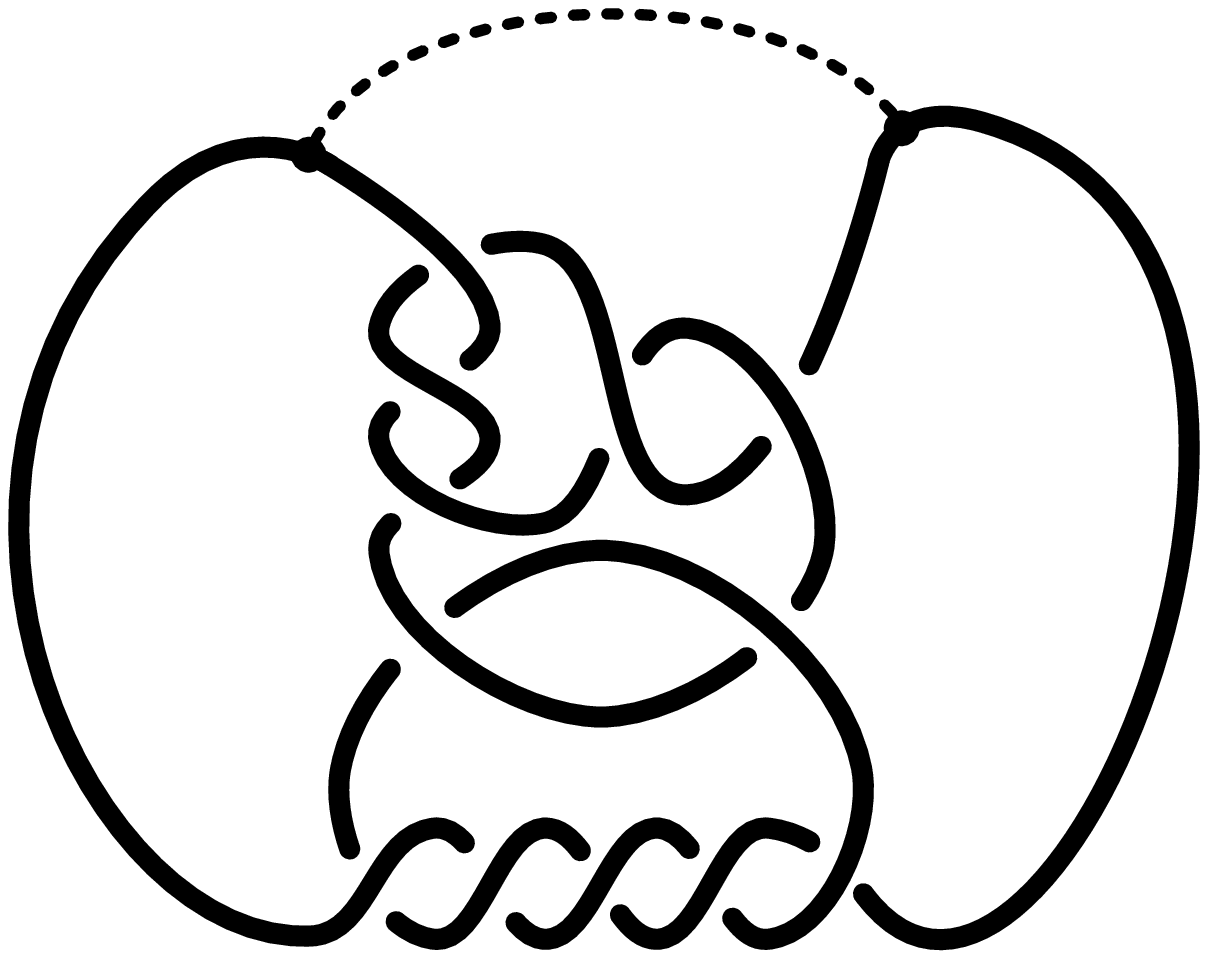}}
\end{center}
\caption{A tangle (left) with Seifert fibered two-fold branched cover (by abuse, the arcs $\gamma_{\overzero}$ and $\gamma_{0}$ have been labelled by their respective lifts), the branch set associated to $7\alpha+3\fibre$ Dehn filling giving rise to the knot $12^n_{500}$ (centre), and another view of $12^n_{500}$ with dashed arc $\gamma$ that lifts to a knot in the cover $\widetilde{\gamma}$  isotopic to a regular fibre.}
\label{fig:12n500-step-one}\end{figure}
First consider the tangle $T=(B^3,\tau)$ shown in \fullref{fig:12n500-step-one}. The Seifert fibration in the cover has base orbifold $D^2(2,5)$ (notice that the tangle is the sum of two rational tangles). Moreover, this is a quasi-alternating framing for $T$: $\tau(0)$ is a connect sum of two-bridge links (the Hopf link and the cinqfoil) with $\det(\tau(0))=(2)(5)$, while $\tau(\overzero)$ is a two-bridge knot (the trefoil) with $\det(\tau(\overzero))=3$. Indeed, one may check that $\tau(1)$ is the knot $7_3$ with $\det(\tau(1))=13=3+(5)(2)$ as required.

As an application of \fullref{thm:main-qa}, it follows that $\tau(r)$ is a quasi-alternating link for every $r\ge0$. For example, the quasi-alternating knot $\tau(\frac{7}{3})\simeq12^n_{500}$ is shown in \fullref{fig:12n500-step-one}, noting that $\frac{7}{3}=[2,3]$. This gives the branch set associated to Dehn filling $\partial\left(\Br(B^3,\tau)\right)$ along the slope $7\alpha+3\fibre$. Here, $\det(\tau(\frac{7}{3}))=7(3)+3(5)(2)=51$ which decomposes as $\det(\tau(\frac{5}{2}))+\det(\tau(2))=(5(3)+2(10))+(2(3)+1(10))=35+16$ by resolving the final crossing added by the continued fraction. Note that one must verify that $c_M=1$, though this may be easily determined from $\det(\tau(0))$ and $\det(\tau(\overzero))$.

\begin{figure}[ht!]
\begin{center}
\labellist \small
	\pinlabel $\alpha'$ at 367 505
	\pinlabel $\fibre'$ at 213 290
\endlabellist
\raisebox{0pt}{\includegraphics[scale=0.25]{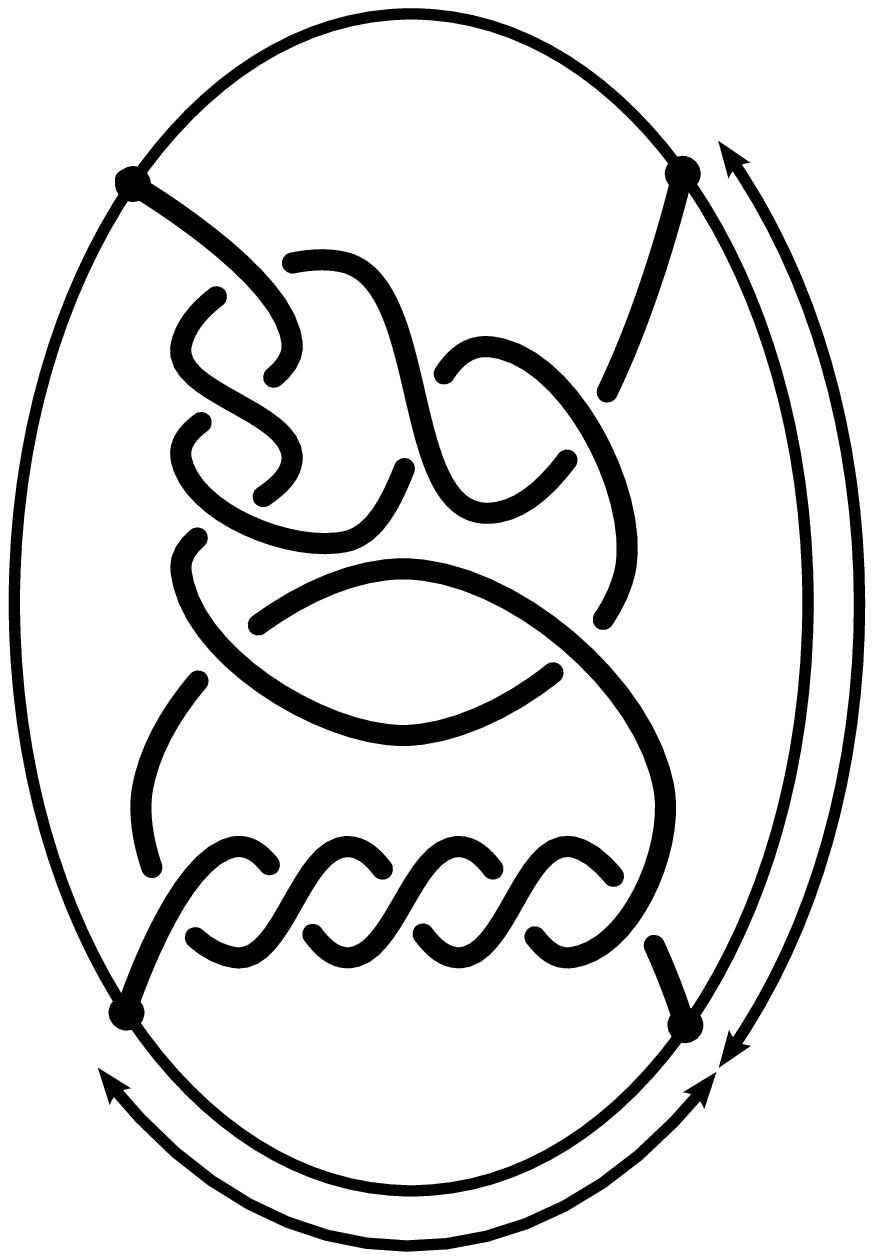}}\qquad\qquad
\raisebox{0pt}{\includegraphics[scale=0.25]{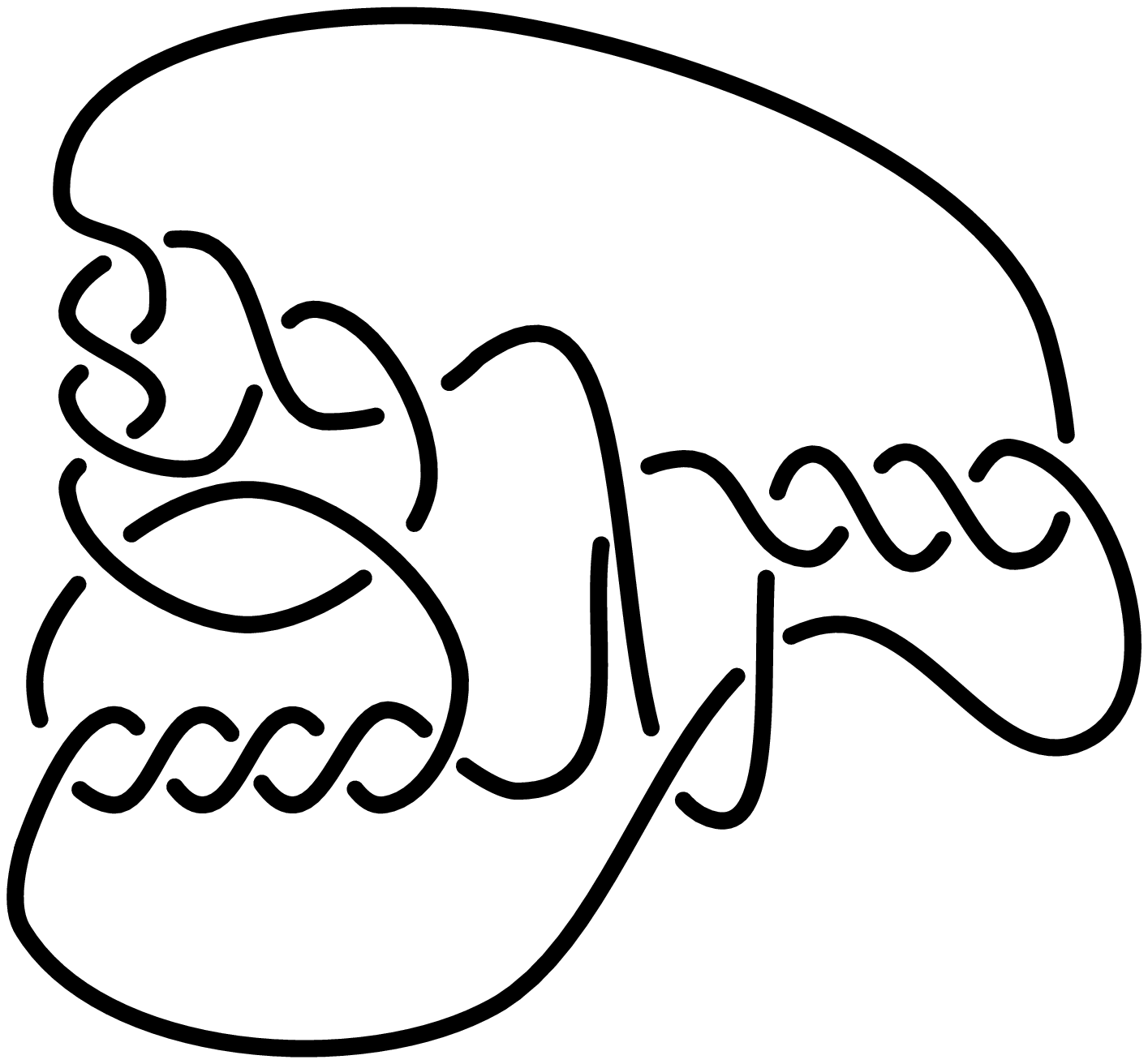}}\qquad\qquad
\end{center}
\caption{A quasi-alternating tangle (left) and a quasi-alternating Montesinos knot (right).}
\label{fig:12n500-step-two}\end{figure}

Given that $12^n_{500}$ is a quasi-alternating knot, with small Seifert fibered two-fold branched cover, we may repeat the above process forming a new quasi-alternating tangle $T'=(B^3,\tau')$. By removing a neighbourhood of the arc $\gamma$ shown in \fullref{fig:12n500-step-one} we obtain the framed quasi-alternating tangle shown in \fullref{fig:12n500-step-two}. By construction, $\tau(0)$ gives a connect sum of two-bridge knots (with $\det(\tau(0))=(2)(5)(7)$) and branch set for Dehn filling along $\fibre'$, while $\tau(\overzero)\simeq12^n_{500}$ is the branch set for Dehn filling along $\alpha'$. In fact, $\gamma$ lifts to a knot isotopic to $\fibre$ in $\Br(S^3, 12^n_{500})$.

Again, every link $\tau'(r)$ is quasi-alternating for $r\ge0$ as a result of \fullref{thm:main-qa}. A particular example, corresponding to filling along the slope $13\alpha'+9\fibre'$, is shown in \fullref{fig:12n500-step-two}. It may be easily verified that $\frac{13}{9}=[1,2,4]$ and $\det(\tau(\frac{13}{9}))=13(51)+9(2)(5)(7)=1293$

This process may now be iterated {\em ad infinitum} to obtain further infinite families of quasi-alternating Montesinos links. We remark that every quasi-alternating Montesinos link $L$ is contained in such an infinite family: it suffices to identify an embedded arc $\gamma$ with endpoints on $L$ whose lift in $\Br(S^3,L)$ is isotopic to a regular fibre, and repeat the construction above. As a result, every quasi-alternating Montesinos link arises through this iterative construction.

\bibliographystyle{gtart}
\bibliography{references/published,references/unpublished}

\end{document}